\documentclass[12pt,twoside,reqno]{article}

\usepackage{amsfonts, amsthm, amsmath, amssymb}
\usepackage{hyperref}
\usepackage{mathrsfs}
\usepackage{indentfirst}
\usepackage{bigints}
\hypersetup{colorlinks=false}

\usepackage[margin=1.125in]{geometry}

\numberwithin{equation}{section}

\allowdisplaybreaks[4]

\newtheorem{thm}{Theorem}[section]
    \newtheorem{cor}[thm]{Corollary}
    \newtheorem{lem}[thm]{Lemma}

    \theoremstyle{definition}

    \numberwithin{equation}{section}
\begin{document}
\title{The exceptional set for Diophantine inequality with mixed powers of primes}
\author{Yu Fu, Linzhu Fu, Liqun Hu\\
Department of Mathematics, Nanchang University,\\
Nanchang, Jiangxi 330031, P.R. China\\
E-mail:fuyu7812@163.com\\ fulinzhu2024@163.com\\and  huliqun@ncu.edu.cn }
\date{}
\maketitle
{\bf Abstract}:
    Assume that $\lambda _1, \lambda _2, \lambda_3,\lambda_4,\lambda_5,\lambda_6,\lambda_7$ are non-zero real numbers , $\lambda _1/\lambda _2$ is an irrational number. Let $\mathcal{V}   $  be a well-spaced sequence, and $\delta >0$. For any given positive integer $k\geq 5$ and any $\varepsilon >0$, we give the upper bound of the number of $\upsilon \in \mathcal{V}   $ with $\upsilon \leq X$ for which the inequality 
    $$
    \left | \lambda_1p_1^2+\lambda _2p_2^3+\lambda _3p_3^3+\lambda _4p_4^3+\lambda _5p_5^3+\lambda _6p_6^4+\lambda _7p_7^k-\upsilon  \right | <{\upsilon}^{-\delta}
    $$ 
    has no solution in primes $p_1,  p_2, p_3, p_4, p_5, p_6, p_7$.

\textbf{keywords}: Exceptional set$\cdot$ Sieve functions$\cdot$ Diophantine inequality$\cdot$ Prime

\section{Introduction}

Let $\lambda_1,\lambda_2,\lambda_3,\lambda_4,\lambda_5,\lambda_6,\lambda_7$ be non-zero real numbers, not all negative and the ratio of at least two numbers is irrational. In 1994, Yu \cite{yhb} considered the inequality
$$
 \left| \lambda_1x_1^2+\lambda _2x_2^3+\lambda _3x_3^3+\lambda _4x_4^3+\lambda _5x_5^3+\lambda _6x_6^4+\lambda _7x_7^4+\eta\right| < (\max_{1\leq j\leq7} \left | x_j  \right | )^{-\sigma}
$$
and proved that the inequality has infinite integer solutions  $x_j$ $(1\leq j\leq7)$ with arbitrary real number $\eta$  for $0<\sigma<1/36$. In 2005, Li \cite{lwp} proved the more general Diophantine inequality
\begin{equation}
\left| \lambda_1x_1^2+\lambda _2x_2^3+\lambda _3x_3^3+\lambda _4x_4^3+\lambda _5x_5^3+\lambda _6x_6^a+\lambda _7x_7^b+\eta\right| < (\max_{1\leq j\leq7} \left | x_j  \right | )^{-\sigma}
\end{equation}
 has infinite integer solutions  $x_j$ $(1\leq j\leq7)$  with arbitrary real number $\eta$ for a=4, b=5 and $0<\sigma<1/30$. Li and Gong \cite{lwp2} proved that for $a=3$, $4\leq b\leq11$ or $a=4$, $4\leq b \leq5$, the inequality (1.1) has infinitely many  positive integer solutions  $x_j$ $(1\leq j\leq7)$ for some $\sigma=\sigma(a,b)>0.$ Moreover, when $a=4$, $b\geq5$, we let $b = k$, Xi and Mu \cite{xly} proved the inequality
 \begin{equation}
 	\left| \lambda_1x_1^2+\lambda _2x_2^3+\lambda _3x_3^3+\lambda _4x_4^3+\lambda _5x_5^3+\lambda _6x_6^4+\lambda _7x_7^k+\eta\right| < (\max_{1\leq j\leq7} \left | x_j  \right | )^{-\sigma}\notag
 \end{equation}
  has infinitely many positive integer solutions  $x_j$ $(1\leq j\leq7)$  with arbitrary real number $\eta$ for $a=4$, $k\in$\{5, 6\} and $0<\sigma<\frac{1}{12(k-3)}$.
 We call an increasing sequence $\upsilon_1<\upsilon_2<\cdot\cdot\cdot$ of positive real numbers a well-spaced sequence if there exist positive constants $C>c>0$ such that 
 $$
 0<c<{\upsilon}_{i+1}-{\upsilon}_{i}<C,\enspace i=1, 2,\dots.
 $$
 Then, in this paper,  we consider the exceptional set for the inequality
$$
\left | \lambda_1p_1^2+\lambda _2p_2^3+\lambda _3p_3^3+\lambda _4p_4^3+\lambda _5p_5^3+\lambda _6p_6^4+\lambda _7p_7^k-\upsilon  \right | <{\upsilon}^{-\delta}, 
$$
when $k\in \mathbb{Z},$ $k\geq5,$    $\upsilon\in\mathcal{V} $ and $\mathcal{V}$ is a  well-spaced sequence.

We prove the following theorem.

$\mathbf{Theorem 1.1} $
Suppose that $k\ge 5$ is a positive integer, $\lambda _1,\lambda _2,\lambda_3,\lambda_4,\lambda_5,\lambda_6,\lambda_7$ are positive real numbers, and $\lambda _1/\lambda _2$ is an irrational and algebraic number. Assume that $\mathcal {V} $ is a well-spaced sequence and $\delta >0$. For any given $\varepsilon >0$ , the number $E_k(\mathcal{V} , X, \delta)$ of $\upsilon \in \mathcal{V} $ with $\upsilon \leq X$ for which the inequality 
\begin{equation}\label{bbb}
 \left | \lambda_1p_1^2+\lambda _2p_2^3+\lambda _3p_3^3+\lambda _4p_4^3+\lambda _5p_5^3+\lambda _6p_6^4+\lambda _7p_7^k-\upsilon  \right | <{\upsilon}^{-\delta}
\end{equation}
has no solution within the primes $p_1, p_2, p_3, p_4, p_5, p_6, p_7$,  satisfies
\begin{equation}\label{eee}
    E_k(\mathcal{V}, X, \delta)\ll X^{\sigma(k) +2\delta  +\varepsilon },\nonumber
\end{equation}
where
\begin{equation}\label{bbb}
\sigma (k)=\frac{44}{45} -\frac{24}{25m_1(k)},
\end{equation}
and
\begin{equation}\label{aaa}
     m_1(k)= 
\begin{cases} 
  2\lceil343k/50-20\rceil, & \text{if $k=5$,}\\  
  2\lceil (49k/50-\lfloor 49k/50\rfloor/2)(\lfloor 49k/50\rfloor+1)\rceil, 
  & \text{if $k\geq6$.}\nonumber
\end{cases} 
\end{equation}

{\bf Remark.}
Compared with  the previous work, we consider the prime solution of the problem with a well-spaced sequence and explore the exception set of the inequality (1.2).

To prove Theorem 1.1, we adopt the circle method from Davenport-Heilbronn \cite{DH} and combine it with the sieve method. Our paper addresses the case of exceptional
sets, utilizing a more precise estimate of the integral through the optimal application of Hölder’s inequalities, as detailed in \cite{qyy}.

The following are the components of this paper. In Sect.2, we give the outline of the proof of Theorem 1.1. In Sects.3-5, we estimate the major arc, the minor arc and the trivial arc. In the end, in Sect.6, we complete the proof of Theorem 1.1.

{\bf Notations.}
 In this paper, the letter $p$, with or without subscript, represents a prime number. The letter $\varepsilon $ represents a sufficiently small positive number, The value of it may vary in different cases. Let $\left \lceil x \right \rceil $ denote the smallest integer not less than $x$, $\left \lfloor x \right \rfloor $ denote the largest integer not exceeding $x$. We write $e(x)=\exp (2\pi ix )$ and abbreviate $\log X$ to $L$.

\section{Outline of the method  }

Let $k\ge 5$ be an integer and $\eta$ be a fixed small positive number. Let $0<\tau<1$. Denote 
$$
K_{ \tau}(\alpha )=\left ( \frac{\sin \pi\tau\alpha   }{\pi \alpha }  \right ) ^2
$$
for $\tau>0$ and $\alpha\ne 0$. By continuity, $K_\tau(0)={\tau^2}$.
Then
\begin{equation}\label{bbb}
K_\tau (\alpha )\ll \min (\tau ^2,\left | \alpha  \right |^{-2} ).
\end{equation}
According to the results of the Fourier transform, we can get
\begin{equation}\label{bbb}
\hat{K} _\tau (t)=\int _\mathbb{R}{K_\tau (\alpha )e(t\alpha )}d\alpha =\max(0,\tau -\left | t \right | ).
\end{equation}
Assume that  $a/q$ is a convergent to $\lambda _1/\lambda _2$ as the denominator $q$ tends to infinity, we can deduce the following conclusion under fixed values of $X=q^{\frac{126}{101} }$.

Write
$$
I_1=[(\eta X)^{\frac{1}{2} },X^{\frac{1}{2} }],\enspace I_2=[(\eta X)^{\frac{1}{3} },X^{\frac{1}{3} }],\enspace I_3=[(\eta X)^{\frac{1}{4} },X^{\frac{1}{4} }],\enspace I_4=[(\eta X)^{\frac{1}{k} },X^{\frac{1}{k} }],
$$
and $\varOmega=I_1\times I_2 \times I_3 \times I_4$.
We use the function $\rho(m)$ defined in Harman and Kumchev \cite{H}(see also \cite[Sect.8]{G}). Write
$$
\psi (m,z)=
\begin{cases}
1, & \text{if $p\mid m$$\Rightarrow$ $p\ge z$,} \\
0,& \text{otherwise.}\\
\end{cases}
$$
The function $\rho(m)$ takes the form
$$
\rho (m)=\psi (m,{X^{5/42}})-\sum_{X^{5/42} \leq p\leq {X^{1/4} } }\psi (m/p,z(p)),
$$
where
$$
z(p)=
\begin{cases} 
  X^{5/28}p^{-1/2}, & \text{if $p< X^{1/7}$,}\\  
  p, & \text{if $X^{1/7}\leq p\leq X^{3/14}$,}\\  
  X^{5/14}p^{-1}, 
  & \text{if $p>X^{3/14}$.}
\end{cases}
$$
For $m\leq X^{\frac{1}{2}},$ it follows from the construction of $\rho(m)$ that (see \cite[(2.3)]{H})
$$
\rho(m)\leq
\begin{cases}
	1, & \text{if $m$ is prime, } \\
	0,& \text{otherwise.}\\
\end{cases}
$$
Let $I'$ be any subinterval of $I_1.$ Then we have
\begin{equation}\label{bbb}
\sum_{m\in I'} \rho (m)=\kappa \left |I'  \right | L^{-1}+  O(X^\frac{1}{2} L^{-2}),
\end{equation}
where $\kappa>0$ is an absolute constant. In addition, we write
\begin{align*}
S_2(\alpha ,\rho  )&=\sum_{m_1\in I_1} \rho (m_1) e(m_1^2\alpha ),\notag\\
S_3(\alpha )&=\sum_{p_j\in I_2} \log p_j e(p_j^3\alpha ),\enspace j=2,3,4,5,\notag\\
S_4(\alpha )&=\sum_{p_6\in I_3} \log p_6 e(p_6^4\alpha ),\notag\\
S_k(\alpha )&=\sum_{p_7\in I_4} \log p_7 e(p_7^k\alpha).
\end{align*}
For any measurable subset $\mathfrak{X} $ of $\mathbb{R}$, let
$$
I(\tau ,\upsilon,\mathfrak{X} ,\rho )=
\bigintsss_{\mathfrak{X} }S_2(\lambda _1\alpha,\rho )\prod\limits_{j=2}^5S_3(\lambda _j\alpha )S_4(\lambda _6\alpha )S_k(\lambda _7\alpha )K_\tau (\alpha )e(-\upsilon \alpha )d\alpha .
$$
Then
\begin{align}
     {I}&(\tau ,\upsilon ,\mathbb{R} ,\rho ) \nonumber\\
     =&\sum\limits_{(m_1,p_2,p_3,p_4,p_5,p_6,p_7)\in \Omega}\rho (m_1)\log p_2\log p_3\log p_4\log p_5\log p_6\log p_7  \nonumber\\
    & \times {\bigintsss_{\mathbb{R}}K_\tau (\alpha ) e((\lambda _1m_1^{2}+\lambda _2p_2^{3}+\lambda _3p_3^{3}+\lambda _4p_4^{3}+\lambda _5p_5^{3}+\lambda _6p_6^{4}+\lambda _7p_7^{k}-\upsilon )\alpha) }d\alpha \nonumber\\
    =&\sum\limits_{(m_1,p_2,p_3,p_4,p_5,p_6,p_7)\in \Omega}\rho (m_1)\log p_2\log p_3\log p_4\log p_5\log p_6\log p_7 \nonumber\\
    & \times \max(0,\tau -\left |\lambda _1m_1^{2}+\lambda _2p_2^{3}+\lambda _3p_3^{3}+\lambda _4p_4^{3}+\lambda _5p_5^{3}+\lambda _6p_6^{4}+\lambda _7p_7^{k}-\upsilon\right | )\nonumber\\
    \leq & \tau L^6\mathcal{N} (x),\nonumber
\end{align}
where $\mathcal{N} (x)$ represents the number of solutions of inequality 
$$
\left |\lambda _1p_1^{2}+\lambda _2p_2^{3}+\lambda _3p_3^{3}+\lambda _4p_4^{3}+\lambda _5p_5^{3}+\lambda _6p_6^{4}+\lambda _7p_7^{k}-\upsilon   \right |<\tau
$$
with $(p_1,p_2,p_3,p_4,p_5,p_6,p_7)\in \Omega$.

To estimate the integral $I(\tau ,\upsilon ,\mathbb{R}   ,\rho ) $  as needed, we divide the whole interval into major arc $\mathfrak{M}$, minor arc $\mathfrak{m}$ and trivial arc $\mathfrak{t}$. We define
$$
\mathfrak{M} =\left [ -\frac{P}{X} , \frac{P}{X} \right ] ,\enspace
\mathfrak{m} =\bigg( -R , -\frac{P}{X}\bigg) \cup \bigg(\frac{P}{X},R\bigg),\enspace
\mathfrak{t} =\mathbb{R}\backslash(\mathfrak{M\cup m} ),
$$
where $P=X^{{\frac{1}{4}}-\varepsilon }$, $R=\tau ^{-2}X^{\frac{5k-12}{32k} }L^3.$ Thus,
$$
I(\tau ,\upsilon ,\mathbb{R}   ,\rho )=I(\tau ,\upsilon ,\mathfrak{M}   ,\rho )+I(\tau ,\upsilon ,\mathfrak{m}   ,\rho )+I(\tau ,\upsilon ,\mathfrak{t}   ,\rho ).
$$

The main contribution to the integral $I(\tau,\upsilon,\mathbb{R},\rho)$ comes from the major arc. In order to handle the integral on the major arc, we need some lemmas. First, we define
for $r\geq1$, 
$$
J_r(X,h)={\bigintsss_{\varepsilon X}^X \bigg( \theta\Big((x+h)^{\frac{1}{r}}\Big)-\theta(x^\frac{1}{r})-\Big((x+h)^{\frac{1}{r}}-x^{\frac{1}{r}}\Big)\bigg)^2} d\alpha
$$
where $\theta(x)=\sum_{p\leq x}\log p$ is the Chebyshev function, and  
\begin{align*}
	U_3(\alpha)=\sum_{n\in I_2} e(n^3\alpha ),\enspace
	U_4(\alpha )=\sum_{n\in I_3} e(n^4\alpha ),\enspace
	U_k(\alpha )=\sum_{n\in I_4} e(n^k\alpha).
\end{align*}
The following Lemmas 2.1 and 2.2 are Theorems 3.1 and 3.2 in \cite{666}, respectively.
\begin{lem}
	Let $r\geq1$ be a real number. For $0\leq Y\leq1/2$ we have 
$$
{\bigintsss_{-Y}^Y \left| S_r(\alpha)-U_r(\alpha)\right|^2}d\alpha\ll\frac{X^{\frac{2}{r}-2}\log^2X}{Y}+Y^2X+Y^2J_r\bigg(X,\frac{1}{2Y}\bigg).
$$
\end{lem}
\begin{lem}
	Let $r\geq1$ be a real number. There exists a positive constant $c_1$=$c_1(\varepsilon)$, which doesn't depend on r, such that
$$
J_r(X,h)\ll h^2 X^{\frac{2}{r}-1} \exp \bigg(-c_1\bigg(\frac{\log X}{\log \log X}\bigg)^{\frac{1}{3}}\bigg)
$$
uniformly for $X^{1-\frac{5}{6r}+\varepsilon}\leq h\leq X$.
\end{lem}
 Combining Lemmas 2.1 and 2.2, we can deduce the following lemma easily.
\begin{lem}
	Let $r \geq1$ be a real number. For any fixed real number $A\geq6$, we have
$$
{\bigintsss_{\left|\alpha\right| \leq X^{\frac{5}{6r}-1-\varepsilon}} \left| S_r(\alpha)-U_r(\alpha)\right|^2}d\alpha\ll X^{\frac{2}{r}-1}(\log X)^{-A}.
$$
\end{lem}
	 
\section{The major arc}
In this section, we estimate the contribution of the major arc. Let $\phi=X^{-1+\frac{1}{2k}-\varepsilon}$. We first consider the region $\mathfrak{M}^{\textasteriskcentered} \subseteq \mathfrak{M}$ with $\mathfrak{M}^{\textasteriskcentered} = \left\lbrace \alpha:\left| \alpha\right| \leq\phi\right\rbrace .$ Write
\begin{align}
T_2(\alpha)=&\bigintsss_{I_1}e(t^2\alpha)dt,\enspace T_3(\alpha)=\bigintsss_{I_2}e(t^3\alpha)dt,\notag\\
T_4(\alpha)=&\bigintsss_{I_3}e(t^4\alpha)dt,\enspace
T_k(\alpha)=\bigintsss_{I_4}e(t^k\alpha)dt.
\end{align}
  By Prime Number Theorem and the first derivative estimate for trigonometric integrals (see \cite{T}), we have
\begin{eqnarray}
  	S_2(\alpha,\rho )\ll X^{\frac{1}{2}}L^{-1}, & &T_2(\alpha)\ll X^{\frac{1}{2}-1}\min (X, |\alpha|^{-1}),\notag\\
  S_3(\alpha )\ll X^{\frac{1}{3}}, & &T_3(\alpha)\ll X^{\frac{1}{3}-1}\min (X, |\alpha|^{-1}),\notag\\
  S_4(\alpha )\ll X^{\frac{1}{4}}, & &T_4(\alpha)\ll X^{\frac{1}{4}-1}\min (X, |\alpha|^{-1}),\notag\\
  S_k(\alpha )\ll X^{\frac{1}{k}}, & &T_k(\alpha)\ll X^{\frac{1}{k}-1}\min (X, |\alpha|^{-1}). 
\end{eqnarray}
It follows that 
\begin{align}
    {I}&(\tau,\upsilon ,\mathfrak{M}^{\textasteriskcentered} ,\rho ) \nonumber\\
    =& \bigintsss_{\mathfrak{M}^{\textasteriskcentered}}S_2(\lambda _1\alpha ,\rho )\prod\limits_{j=2}^5S_3(\lambda _j\alpha )S_4(\lambda _6\alpha )K_\tau (\alpha )e(-\upsilon \alpha )d\alpha \nonumber\\
    =& \kappa L^{-1}\bigintsss_{\mathfrak{M}^{\textasteriskcentered} }T_2(\lambda _1\alpha )\prod\limits_{j=2}^5T_3(\lambda _j\alpha )T_4(\lambda _6\alpha )T_k(\lambda _7\alpha )K_\tau (\alpha )e(-\upsilon \alpha )d\alpha  \nonumber\\
    & +\bigintsss_{\mathfrak{M}^{\textasteriskcentered} }(S_2(\lambda _1\alpha ,\rho )-\kappa L^{-1}T_2(\lambda _1\alpha))\prod\limits_{j=2}^5T_3(\lambda _j\alpha)T_4(\lambda _6\alpha )
    T_k(\lambda _7\alpha )K_\tau (\alpha )e(-\upsilon \alpha )d\alpha  \nonumber\\
    & + \bigintsss_{\mathfrak{M}^{\textasteriskcentered} }S_2(\lambda _1\alpha ,\rho )(S_3(\lambda _2\alpha )-T_3(\lambda _2\alpha ))\prod\limits_{j=3}^5T_3(\lambda _j\alpha)T_4(\lambda _6\alpha )T_k(\lambda _7\alpha )K_\tau (\alpha )e(-\upsilon\alpha )d\alpha  \nonumber\\
    & + \bigintsss_{\mathfrak{M}^{\textasteriskcentered}}S_2(\lambda _1\alpha ,\rho )S_3(\lambda _2\alpha )(S_3(\lambda _3\alpha )-T_3(\lambda _3\alpha ))\prod\limits_{j=4}^5T_3(\lambda _j\alpha)T_4(\lambda _6\alpha )T_k(\lambda _7\alpha )K_\tau (\alpha )e(-\upsilon \alpha )d\alpha \nonumber\\
    & + \bigintsss_{\mathfrak{M}^{\textasteriskcentered} }S_2(\lambda _1\alpha ,\rho )\prod\limits_{j=2}^3S_3(\lambda _j\alpha )(S_3(\lambda _4\alpha )-T_3(\lambda _4\alpha ))T_3(\lambda _5\alpha)T_4(\lambda _6\alpha )T_k(\lambda _7\alpha )K_\tau (\alpha )e(-\upsilon \alpha )d\alpha \nonumber\\
    & + \bigintsss_{\mathfrak{M}^{\textasteriskcentered} }S_2(\lambda _1\alpha ,\rho )\prod\limits_{j=2}^4S_3(\lambda _j\alpha )(S_3(\lambda _5\alpha )-T_3(\lambda _5\alpha ))T_4(\lambda _6\alpha )T_k(\lambda _7\alpha )K_\tau (\alpha )e(-\upsilon \alpha )d\alpha \nonumber\\
     & + \bigintsss_{\mathfrak{M}^{\textasteriskcentered}}S_2(\lambda _1\alpha ,\rho )\prod\limits_{j=2}^5S_3(\lambda _j\alpha )(S_4(\lambda _6\alpha )-T_4(\lambda _6\alpha ))T_k(\lambda _7\alpha )K_\tau (\alpha )e(-\upsilon \alpha )d\alpha \nonumber\\
     & + \bigintsss_{\mathfrak{M}^{\textasteriskcentered} }S_2(\lambda _1\alpha ,\rho )\prod\limits_{j=2}^5S_3(\lambda _j\alpha )S_4(\lambda _6\alpha )(S_k(\lambda _7\alpha )-T_k(\lambda _7\alpha ))K_\tau (\alpha )e(-\upsilon \alpha )d\alpha \nonumber\\
    =& J_1+J_2+J_3+J_4+J_5+J_6+J_7+J_8.\nonumber
\end{align}
In the following, we will show that $J_1\gg\tau^2 L^{-1}X^{\frac{13}{12}+\frac{1}{k}}$ and $J_i=o({\tau}^2 L^{-1} X^{\frac{13}{12}+\frac{1}{k}})$ for $i=2,3,4,5,6,7,8$.
Since the estimates for $J_4,J_5,J_6$ are similar to $J_3$, so we restrict our attention to estimate $J_1,J_2,J_3,J_7,J_8$.
\subsection{Lower bound for $J_1$}
We first establish the lower bound for $J_1$. Note that
\begin{eqnarray}
    J_1 &=& \kappa L^{-1}\bigintsss_{\mathfrak{M}^{\textasteriskcentered} }T_2(\lambda _1\alpha )\prod\limits_{j=2}^5T_3(\lambda _j\alpha )T_4(\lambda _6\alpha )T_k(\lambda _7\alpha )K_\tau (\alpha )e(-\upsilon \alpha )d\alpha \notag \\
    &=&\kappa L^{-1}\bigintsss_{\mathbb {R} }T_2(\lambda _1\alpha )\prod\limits_{j=2}^5T_3(\lambda _j\alpha )T_4(\lambda _6\alpha )T_k(\lambda _7\alpha )K_\tau (\alpha )e(-\upsilon \alpha )d\alpha \notag\\
    & &+O\bigg(L^{-1}\bigintsss_{\phi }^{\infty}\left|T_2(\lambda _1\alpha )\prod\limits_{j=2}^5T_3(\lambda _j\alpha )T_4(\lambda _6\alpha )T_k(\lambda _7\alpha )\right|K_\tau (\alpha )d\alpha\bigg).
\end{eqnarray}
Combining this with (2.1) and (3.2), we obtain that the error term in (3.3) satisfies
$$
\ll {\tau}^2 L^{-1} X^{\frac{1}{k}-\frac{59}{12}}\bigintsss_{\phi }^{+\infty}\frac{d\alpha}{{\alpha}^7}
\ll {\tau}^2 L^{-1} X^{\frac{13}{12}+\frac{1}{k}}X^{-\frac{3}{k}}
=o({\tau}^2 L^{-1} X^{\frac{13}{12}+\frac{1}{k}}).
$$
Hence, we get 
$$
J_1=\kappa L^{-1}\bigintsss_{\mathbb {R} }T_2(\lambda _1\alpha )\prod\limits_{j=2}^5T_3(\lambda _j\alpha )T_4(\lambda _6\alpha )T_k(\lambda _7\alpha )K_\tau (\alpha )e(-\upsilon \alpha )d\alpha + o({\tau}^2 L^{-1} X^{\frac{13}{12}+\frac{1}{k}}).
$$
For convenience, we write \textbf{y} = $(y_1, y_2,...,y_7).$ Changing the order of integration and substitution variables, we have
\begin{eqnarray}
& &\bigintsss_{\mathbb {R} }T_2(\lambda _1\alpha )\prod\limits_{j=2}^5T_3(\lambda _j\alpha )T_4(\lambda _6\alpha )T_k(\lambda _7\alpha )K_\tau (\alpha )e(-\upsilon \alpha )d\alpha \notag \\
&=&\frac{1}{648k(\lambda_1)^\frac{1}{2}(\lambda_2\lambda_3\lambda_4\lambda_5)^\frac{2}{3}(\lambda_6)^\frac{3}{4}(\lambda_7)^\frac{1-k}{k}}\bigintsss(y_1)^{-\frac{1}{2}}(y_2y_3y_4y_5)^{-\frac{2}{3}}(y_6)^{-\frac{3}{4}}(y_7)^{\frac{1-k}{k}}\notag \\
 &  &\times\bigintsss_{\mathbb {R} }e\bigg(\Big(\sum\limits_{j=1}^7y_j-\upsilon\Big)\alpha \bigg)K_\tau(\alpha)d\alpha d\textbf{y}\notag \\
&\gg&\bigintsss(y_1)^{-\frac{1}{2}}(y_2y_3y_4y_5)^{-\frac{2}{3}}(y_6)^{-\frac{3}{4}}(y_7)^{\frac{1-k}{k}}\hat{K} _\tau\bigg(\sum\limits_{j=1}^7y_j-\upsilon \bigg)d\textbf{y},\notag
\end{eqnarray}           
where the domain of integration for $\textbf{y}=(y_1, y_2,...,y_7)$ are $[\lambda_i\eta X,\lambda_i X]$, $1\leq i\leq 7$, respectively, and the integral satisfies $\left|\sum\limits_{j=1}^7y_j-\upsilon\right|<\tau.$ Taking $\left|\sum\limits_{j=1}^7y_j-\upsilon\right|<\frac{\tau}{2}$, and noting that $\lambda_7\eta X\leqslant y_7\leqslant\lambda_7 X$, then we can get a lower bound for the last integral
\begin{eqnarray}
&\gg&\tau X^{\frac{1-k}{k}}\bigintsss_{\lambda_1\eta X}^{\lambda_1 X}\cdot\cdot\cdot \bigintsss_{\lambda_6\eta X}^{\lambda_6 X}\bigintsss_{\sum\limits_{j=1}^6y_j-\upsilon-\frac{\tau}{2}}^{\sum\limits_{j=1}^6y_j-\upsilon+\frac{\tau}{2}}(y_1)^{-\frac{1}{2}}(y_2y_3y_4y_5)^{-\frac{2}{3}}(y_6)^{-\frac{3}{4}}\notag\\ & & \times \hat{K} _\tau \bigg(\sum\limits_{j=1}^7y_j-\upsilon\bigg)dy_1\cdot\cdot\cdot dy_7  \notag \\
&\gg&\tau^2 X^{\frac{13}{12}+\frac{1}{k}},\notag
\end{eqnarray} 
which means that
$$
\bigintsss_{\mathbb {R} }T_2(\lambda _1\alpha )\prod\limits_{j=2}^5T_3(\lambda _j\alpha )T_4(\lambda _6\alpha )T_k(\lambda _7\alpha )K_\tau (\alpha )e(-\upsilon \alpha )d\alpha \notag \\
\gg\tau^2 X^{\frac{13}{12}+\frac{1}{k}}.
$$
Thus we have $J_1\gg\tau^2 L^{-1}X^{\frac{13}{12}+\frac{1}{k}}.$
\subsection{Upper bound for $J_2$}
By (2.3), we get
$$
S_2(\lambda _1\alpha ,\rho )-\kappa L^{-1}T_2(\lambda _1\alpha)\ll X^{\frac{1}{2}}L^{-2}\left ( 1+\left | \alpha  \right |X  \right ). 
$$
So we have
\begin{eqnarray}
J_2 &=&\bigintsss_{\mathfrak{M}^{\textasteriskcentered} }(S_2(\lambda _1\alpha ,\rho )-\kappa L^{-1}T_2(\lambda _1\alpha))\prod\limits_{j=2}^5T_3(\lambda _j\alpha)T_4(\lambda _6\alpha )
T_k(\lambda _7\alpha )K_\tau (\alpha )e(-\upsilon \alpha )d\alpha \notag \\
   &\ll& \tau ^2\bigintsss_{\mathfrak{M}^{\textasteriskcentered}} \left |(S_2(\lambda _1\alpha ,\rho )-\kappa L^{-1}T_2(\lambda _1\alpha))  \right |\left |\prod\limits_{j=2}^5T_3(\lambda _j\alpha)\right |\left | T_4(\lambda _6\alpha ) \right | \left | T_k(\lambda _7\alpha ) \right | d\alpha\notag\\
   &\ll&\tau ^2\bigintsss_{\mathfrak{M}^{\textasteriskcentered} } \left |X^{\frac{1}{2}}L^{-2}\left ( 1+\left | \alpha  \right |X  \right )  \right |\left |\prod\limits_{j=2}^5T_3(\lambda _j\alpha)\right |\left | T_4(\lambda _6\alpha ) \right | \left | T_k(\lambda _7\alpha ) \right | d\alpha\notag \\
   &\ll&\tau ^2L^{-2}X^{\frac{1}{2}}\bigintsss_{0}^{\frac{1}{X} }\left |\prod\limits_{j=2}^5T_3(\lambda _j\alpha)\right |\left | T_4(\lambda _6\alpha ) \right | \left | T_k(\lambda _7\alpha ) \right | d\alpha \notag \\
   &  &+\tau ^2L^{-2}X^{\frac{1}{2}} \bigintsss_{\frac{1}{X} }^{\frac{L}{X} }\left |\prod\limits_{j=2}^5T_3(\lambda _j\alpha)\right |\left | T_4(\lambda _6\alpha ) \right | \left | T_k(\lambda _7\alpha )  \right |  d\alpha \notag \\
   &  &+\tau ^2L^{-2}X^{\frac{3}{2}} \bigintsss_{\frac{1}{X} }^{\frac{L}{X} }\left |\alpha \right |\left |\prod\limits_{j=2}^5T_3(\lambda _j\alpha)\right |\left | T_4(\lambda _6\alpha ) \right | \left | T_k(\lambda _7\alpha )  \right |  d\alpha \notag \\
   &  &+\tau ^2L^{-2}X^{\frac{1}{2}} \bigintsss_{\frac{L}{X} }^{\phi }\left |\prod\limits_{j=2}^5T_3(\lambda _j\alpha)\right |\left | T_4(\lambda _6\alpha ) \right | \left | T_k(\lambda _7\alpha )  \right |  d\alpha \notag \\
   &  &+\tau ^2L^{-2}X^{\frac{3}{2}} \bigintsss_{\frac{L}{X} }^{\phi }\left |\alpha \right |\left |\prod\limits_{j=2}^5T_3(\lambda _j\alpha)\right |\left | T_4(\lambda _6\alpha ) \right | \left | T_k(\lambda _7\alpha )  \right |  d\alpha \notag \\
   &\ll& {\tau}^2 L^{-2} X^{\frac{13}{12}+\frac{1}{k}}\notag\\
   &=&o ({\tau}^2 L^{-1} X^{\frac{13}{12}+\frac{1}{k}}).\notag
\end{eqnarray}
\subsection{Upper bounds for $J_3,J_4,J_5,J_6$}
In order to obtain the bound for $J_3$, by Euler's summation formula, we have
$$
U_j(\lambda\alpha)-T_j(\lambda\alpha)\ll 1+ \left |\alpha\right |X,\enspace j\geqslant 3.
$$
Using (2.1), we have
\begin{eqnarray}
J_3&=& \bigintsss_{\mathfrak{M}^{\textasteriskcentered} }S_2(\lambda _1\alpha ,\rho )(S_3(\lambda _2\alpha )-T_3(\lambda _2\alpha ))\prod\limits_{j=3}^5T_3(\lambda _j\alpha)T_4(\lambda _6\alpha )T_k(\lambda _7\alpha )K_\tau (\alpha )e(-\upsilon \alpha )d\alpha  \notag\\
   &\ll&\tau ^2\bigintsss_{\mathfrak{M}^{\textasteriskcentered} }\left |S_2(\lambda _1\alpha ,\rho )\right |\left | S_3({\lambda_2 \alpha   })- T_3({\lambda_2 \alpha   }) \right | \left |\prod\limits_{j=3}^5T_3(\lambda _j\alpha)\right |\left |T_4(\lambda _6\alpha )\right |\left |T_k(\lambda _7\alpha )\right |d\alpha  \notag\\
    &\ll&\tau ^2\bigintsss_{\mathfrak{M}^{\textasteriskcentered} } \left | S_2({\lambda_1 \alpha  ,\rho  })  \right | \left| S_3({\lambda_2 \alpha   })- U_3({\lambda_2 \alpha   })\right| \left |\prod\limits_{j=3}^5T_3(\lambda _j\alpha)  \right |\left| T_4(\lambda _6\alpha )\right|\left |T_k(\lambda _7\alpha )\right|d\alpha \notag\\
    & &+ \tau ^2\bigintsss_{\mathfrak{M}^{\textasteriskcentered} }\left | S_2({\lambda_1 \alpha  ,\rho  })  \right | \left| U_3({\lambda_2 \alpha   })- T_3({\lambda_2 \alpha   })\right| \left |\prod\limits_{j=3}^5T_3(\lambda _j\alpha)  \right |\left| T_4(\lambda _6\alpha )\right|\left |T_k(\lambda _7\alpha )\right|d\alpha \notag\\
    &=&\tau   ^2(A_1+B_1).\notag
\end{eqnarray}
By Cauchy's inequality and Lemma 2.3, we obtain
\begin{eqnarray}
A_1&\ll& X^{\frac{1}{k}+1}\bigg(\bigintsss_{\mathfrak{M}^{\textasteriskcentered}}\left | S_2({\lambda_1 \alpha,\rho}) \right |^4d\alpha    \bigg)^{\frac{1}{4}} \bigg(\bigintsss_{\mathfrak{M}^{\textasteriskcentered} }\left | S_3({\lambda_2 \alpha})- U_3({\lambda_2 \alpha}) \right |^2d\alpha \bigg )^{\frac{1}{2}}\notag \\
&  &\times \bigg(\bigintsss_{\mathfrak{M}^{\textasteriskcentered} }\left | T_4({\lambda_6 \alpha}) \right |^4d\alpha \bigg)^{\frac{1}{4}} \notag \\
   &\ll_A& X^{\frac{13}{12}-\frac{3}{2k}}(\log{X} )^{-\frac{A}{2}}.\notag\\
B_1&\ll&\bigintsss_{0}^{\frac{1}{X}} \left |S_2({\lambda_1 \alpha  ,\rho  })  \right | \left |\prod\limits_{j=3}^5T_3(\lambda _j\alpha) \right |\left| T_4(\lambda _6\alpha )\right|\left |T_k(\lambda _7\alpha )\right|d\alpha\notag\\
   & &+X\bigintsss_{\frac{1}{X}}^{\phi}\left |\alpha\right | \left |S_2({\lambda_1 \alpha  ,\rho  })  \right | \left |\prod\limits_{j=3}^5T_3(\lambda _j\alpha) \right |\left| T_4(\lambda _6\alpha )\right|\left |T_k(\lambda _7\alpha )\right|d\alpha\notag\\
   &\ll& X^{\frac{3}{4}+\frac{1}{k}+\frac{1}{k}-2\varepsilon}L^{-1}. \notag
\end{eqnarray}
For $k\geqslant5$, we have
$$
B_1\ll X^{\frac{19}{20}+\frac{1}{k}-2\varepsilon}L^{-1},
$$
and
$$
J_3= o ({\tau}^2 L^{-1} X^{\frac{13}{12}+\frac{1}{k}}).
$$
As for $J_4$, $J_5$, $J_6$, we use a similar approach and get
$$
J_i= o ({\tau}^2 L^{-1} X^{\frac{13}{12}+\frac{1}{k}}),\enspace i=4,5,6.
$$
\subsection{Upper bounds for $J_7,J_8$}
The computation for $J_7$ and $J_8$ are similar to that for $J_3$, $J_4$. Also by (2.1), we have
\begin{eqnarray}
	J_7&=& \bigintsss_{\mathfrak{M}^{\textasteriskcentered} }S_2(\lambda _1\alpha ,\rho )\prod\limits_{j=2}^5S_3(\lambda _j\alpha )(S_4(\lambda _6\alpha )-T_4(\lambda _6\alpha ))T_k(\lambda _7\alpha )K_\tau (\alpha )e(-\upsilon \alpha )d\alpha \nonumber\\ \notag\\
	&\ll&\tau ^2\bigintsss_{\mathfrak{M}^{\textasteriskcentered} }\left |S_2(\lambda _1\alpha ,\rho )\right | \left |\prod\limits_{j=2}^5 S_3({\lambda_j \alpha   })\right |\left |S_4(\lambda _6\alpha )-T_4(\lambda _6\alpha )\right |\left |T_k(\lambda _7\alpha )\right |d\alpha  \notag\\
	&\ll&\tau ^2\bigintsss_{\mathfrak{M}^{\textasteriskcentered} } \left |S_2(\lambda _1\alpha ,\rho )\right |\left |\prod\limits_{j=2}^5 S_3({\lambda_j \alpha   })\right | \left |S_4(\lambda _6\alpha )-U_4(\lambda _6\alpha )\right | \left |T_k(\lambda _7\alpha )\right |d\alpha \notag\\
	& &+ \tau ^2\bigintsss_{\mathfrak{M}^{\textasteriskcentered} } \left |S_2(\lambda _1\alpha ,\rho )\right |\left |\prod\limits_{j=2}^5 S_3({\lambda_j \alpha   })\right | \left |U_4(\lambda _6\alpha )-T_4(\lambda _6\alpha )\right |\left |T_k(\lambda _7\alpha )\right |d\alpha \notag\\
	&=&\tau   ^2(A_2+B_2). \notag
\end{eqnarray}
By Cauchy's inequality and Lemma 2.3, we obtain
\begin{eqnarray}
A_2&\ll& X^{\frac{1}{k}+\frac{1}{2}} \bigg(\bigintsss_{\mathfrak{M}^{\textasteriskcentered} }\left | S_4({\lambda_6 \alpha})- U_4({\lambda_6 \alpha}) \right |^2d\alpha \bigg )^{\frac{1}{2}}\prod\limits_{j=2}^5\bigg(\bigintsss_{\mathfrak{M} }\left | S_3(\lambda _j\alpha) \right |^8d\alpha \bigg)^{\frac{1}{8}} \notag \\
&\ll_A& X^{\frac{13}{12}+\frac{1}{k}}(\log{X} )^{-\frac{A}{2}}.\notag\\
B_2&\ll&\bigintsss_{0}^{\frac{1}{X}} \left |S_2({\lambda_1 \alpha  ,\rho  })  \right | \left |\prod\limits_{j=2}^5S_3(\lambda _j\alpha) \right |\left |T_k(\lambda _7\alpha )\right|d\alpha\notag\\
& &+  X\bigintsss_{\frac{1}{X}}^{\phi}\left |\alpha\right |\left |S_2({\lambda_1 \alpha  ,\rho  })  \right |  \left |\prod\limits_{j=2}^5S_3(\lambda _j\alpha) \right |\left |T_k(\lambda _7\alpha )\right|d\alpha\notag\\
&\ll& X^{\frac{5}{6}+\frac{1}{k}+\frac{1}{k}-2\varepsilon}L^{-1}.\notag
\end{eqnarray}
For $k\geqslant5$, we have
$$
B_2\ll X^{\frac{31}{30}+\frac{1}{k}-2\varepsilon}L^{-1},
$$
and
$$
J_7= o ({\tau}^2 L^{-1} X^{\frac{13}{12}+\frac{1}{k}}).
$$
Similarly,
\begin{eqnarray}
	J_8&=& \bigintsss_{\mathfrak{M}^{\textasteriskcentered} }S_2(\lambda _1\alpha ,\rho )\prod\limits_{j=2}^5S_3(\lambda _j\alpha)S_4(\lambda _6\alpha )(S_k(\lambda _7\alpha )-T_k(\lambda _7\alpha ))K_\tau (\alpha )e(-\upsilon \alpha )d\alpha \notag\\
	&\ll&\tau ^2\bigintsss_{\mathfrak{M}^{\textasteriskcentered} }\left |S_2(\lambda _1\alpha ,\rho )\right | \left |\prod\limits_{j=2}^5S_3(\lambda _j\alpha) \right |\left |S_4(\lambda _6\alpha )\right |\left |S_k(\lambda _7\alpha )-T_k(\lambda _7\alpha )\right |d\alpha  \notag\\
	&\ll&\tau ^2\bigintsss_{\mathfrak{M}^{\textasteriskcentered} } \left |S_2(\lambda _1\alpha ,\rho )\right | \left |\prod\limits_{j=2}^5S_3(\lambda _j\alpha) \right|\left |S_4(\lambda _6\alpha )\right |\left |S_k(\lambda _7\alpha )-U_k(\lambda _7\alpha )\right |d\alpha \notag\\
	& &+ \tau ^2\bigintsss_{\mathfrak{M}^{\textasteriskcentered} } \left |S_2(\lambda _1\alpha ,\rho )\right | \left |\prod\limits_{j=2}^5S_3(\lambda _j\alpha) \right |\left |S_4(\lambda _6\alpha )\right |\left |U_k(\lambda _7\alpha )-T_k(\lambda _7\alpha )\right |d\alpha \notag\\
	&=&\tau   ^2(A_3+B_3). \notag
\end{eqnarray}
By Cauchy's inequality and Lemma 2.3, we obtain
\begin{eqnarray}
	A_3&\ll& X^{\frac{1}{2}+\frac{1}{4}} \bigg(\bigintsss_{\mathfrak{M}^{\textasteriskcentered} }\left | S_k({\lambda_7 \alpha})- U_k({\lambda_7 \alpha}) \right |^2d\alpha \bigg )^{\frac{1}{2}}\prod\limits_{j=2}^5\bigg(\bigintsss_{\mathfrak{M}^{\textasteriskcentered} }\left |S_3(\lambda _j\alpha) \right |^8d\alpha \bigg)^{\frac{1}{8}} \notag \\
	&\ll_A&  X^{\frac{13}{12}+\frac{1}{k}}(\log{X} )^{-\frac{A}{2}}.\notag\\
	B_3&\ll&\bigintsss_{0}^{\frac{1}{X}} \left |S_2({\lambda_1 \alpha  ,\rho  })  \right | \left|\prod\limits_{j=2}^5S_3(\lambda _j\alpha) \right| \left|S_4(\lambda _6\alpha )\right|d\alpha\notag\\
	& &+X\bigintsss_{\frac{1}{X}}^{\phi}\left |\alpha\right |\left |S_2({\lambda_1 \alpha  ,\rho  })  \right | \left|\prod\limits_{j=2}^5S_3(\lambda _j\alpha) \right|\left |S_4(\lambda _6\alpha )\right|d\alpha\notag\\
	&\ll&X^{\frac{13}{12}+\frac{1}{k}-2\varepsilon}L^{-1}.\notag 
\end{eqnarray}
So we have
$$
J_8= o ({\tau}^2 L^{-1} X^{\frac{13}{12}+\frac{1}{k}}).
$$
Finally, we have $J_1\gg{\tau}^2 L^{-1} X^{\frac{13}{12}+\frac{1}{k}}$ and $J_i= o ({\tau}^2 L^{-1} X^{\frac{13}{12}+\frac{1}{k}})$, $i=2, 3,4,5,6,7,8.\notag\\$
Therefore
\begin{eqnarray}
  I(\tau ,\upsilon ,\mathfrak{M}^{\textasteriskcentered} ,\rho )\gg {\tau}^2 L^{-1} X^{\frac{13}{12}+\frac{1}{k}}.
\end{eqnarray}
In order to handle $\mathfrak{M}\setminus\mathfrak{M}^*$,
we need the following lemmas.
\begin{lem}
	Let $\alpha$ be a real number, and there exist $a \in \mathbb{Z}$, $q \in \mathbb{N}$ with $(a,\enspace q)= 1$ and $\left| \alpha-a/q \right| < q^{-2}.$ Let $k$ be a positive integer with $k\geq 2.$ Then for any $\varepsilon > 0,$ one has
\begin{align}
	\sum_{1\leq p \leq N} (\log p)e(\alpha p^k)\ll N^{1+\varepsilon}(\frac{1}{q}+\frac{1}{N^{1/2}}+\frac{q}{N^k})^{4^{1-k}}.
\end{align}

\end{lem}
\begin{proof}
	The lemma is from Harman \cite{888}.
\end{proof}
\begin{cor}
	Suppose that $k\geq4$ and $X^{-1+\frac{1}{2k}-\varepsilon} < \alpha \leq X^{-\frac{3}{4}}.$ Then 
	\begin{align}
		S_k(\lambda \alpha)\ll X^{\frac{1}{k}-\frac{4^{1-k}}{2k}}
	\end{align}
\end{cor}
\begin{proof}
	We can take $q= [\left|\lambda \alpha^{-1} \right| ],\ a=1$ in (3.5), so that (3.6) can be deduced from (3.5). 
\end{proof}
\begin{lem}
	 We have
$$
\bigintsss_{\left|\alpha\right|\leq X^{-\frac{2}{3}}  }\left|S_2 (\lambda \alpha) \right| ^2 d\alpha \ll 1.
$$
	
\end{lem}
\begin{proof}
	This lemma is from Ge and Zhao \cite{Ge}. Although there exists sieve function $\rho(m)$ in the definition of $S_2(\lambda_1\alpha)$, we can also by similar argument in \cite{Ge} to get this lemma.
\end{proof}

Then by Cauchy's inequality, we conclude that
\begin{eqnarray}
& &\bigintsss_{\mathfrak{M}\setminus\mathfrak{M}^*}\left|S_2(\lambda _1\alpha ,\rho )\prod\limits_{j=2}^5S_3(\lambda _j\alpha )S_4(\lambda _6\alpha )S_k(\lambda _7\alpha )\right|K_\tau (\alpha )d\alpha \notag\\
&\ll&\tau^{2}\left|S_k(\lambda _7\alpha )\right| \left|S_4(\lambda _6\alpha )\right|  \bigg(\bigintsss_{\left|\alpha\right|\leq X^{-\frac{2}{3}} }\left | S_2({\lambda_1 \alpha, \rho})\right |^2d\alpha \bigg )^{\frac{1}{2}}\notag\\
& &\prod\limits_{j=2}^5\bigg(\bigintsss_0^1 \left | S_3(\lambda _j\alpha) \right |^8d\alpha \bigg)^{\frac{1}{8}}\notag\\
&\ll&\tau^{2}X^{\frac{13}{12}+\frac{1}{k}-\frac{4^{1-k}}{2k}}\notag\\
&=&o ({\tau}^2 L^{-1} X^{\frac{13}{12}+\frac{1}{k}}).
\end{eqnarray}

 Therefore, combining (3.4) and (3.7), we have
\begin{eqnarray}
	I(\tau ,\upsilon ,\mathfrak{M} ,\rho )\gg {\tau}^2 L^{-1} X^{\frac{13}{12}+\frac{1}{k}}.
\end{eqnarray}
\section{The minor arc}
In this section, we estimate the contribution of the minor arc. We define
$$
{\mathfrak{m}}_1  =\left \{ \alpha \in \mathfrak{m}: \left | S_2({\lambda_1 \alpha,\rho   }) \right |\le X^{\frac{3}{7}+2\varepsilon}  \right \}
$$
and
$$
{\mathfrak{m}}_2  =\left \{ \alpha \in \mathfrak{m}: \left | S_3({\lambda_2 \alpha}) \right |\le X^{\frac{11}{36}+2\varepsilon}  \right \}. 
$$

\begin{lem}
For $0<\delta<1$, we have
$$
\bigintsss_{\mathbb{R} }\left | S_k(\lambda _7\alpha )  \right | ^{m(k)}K_\tau (\alpha )d\alpha \ll \tau X^{\frac{m(k)}{k}-\delta+\varepsilon  },
$$
where
$$
m(k)=
\begin{cases} 
	2\lceil (k\delta+1-\lfloor k\delta\rfloor)2^{\lfloor k\delta-1\rfloor}\rceil, & \text{if $\lfloor k\delta\rfloor\leq3$,}\\  
	2\lceil7k\delta-20\rceil, & \text{if $\lfloor  k\delta\rfloor=4$,}\\  
	2\lceil (k\delta-\lfloor k\delta\rfloor/2)(\lfloor k\delta\rfloor+1)\rceil, 
	& \text{if $\lfloor  k\delta\rfloor\geq5$.}
\end{cases}
$$
\end{lem}
\begin{proof}
    The lemma is from Li and Wang \cite [Lemma 4.2]{W}.
\end{proof}
\begin{lem}
	We have
	$$
	\bigintsss_{\mathbb{R} }\left | S_k(\lambda _7\alpha )  \right | ^{m_1(k)}K_\tau (\alpha )d\alpha \ll \tau X^{\frac{m_1(k)}{k}-\frac{49}{50}+\varepsilon  },
	$$
	where
	$$
	m_1(k)=
	\begin{cases}   
		2\lceil343k/50-20\rceil, & \text{if $k=5$,}\\  
		2\lceil (49k/50-\lfloor 49k/50\rfloor/2)(\lfloor 49k/50\rfloor+1)\rceil, 
		& \text{if $k\geq6$.}
	\end{cases}
	$$
\end{lem}
\begin{proof}
	By Lemma 4.1, we take $\delta=\frac{49}{50}$, so the proof is completed.
\end{proof}
\begin{lem}
	We have
	$$
	\bigintsss_{\mathbb{R} }\left | S_3(\lambda _2\alpha )  \right | ^{4}\left | S_k(\lambda _7\alpha )  \right | ^{m_2(k)}K_\tau (\alpha )d\alpha \ll \tau X^{\frac{m_2(k)}{k}+\frac{1}{3}+\varepsilon  },
	$$
	where
	$$
	m_2(k)=
	\begin{cases} 
		2\lceil (2k/3+1-\lfloor 2k/3 \rfloor)2^{\lfloor 2k/3\rfloor-1}\rceil, & \text{if $k=5$,}\\  
		2\lceil14k/3-20\rceil, & \text{if $k=6,7$,}\\  
		2\lceil (2k/3-\lfloor 2k/3\rfloor/2)(\lfloor 2k/3\rfloor+1)\rceil, 
		& \text{if $k\geq8$.}
	\end{cases}
	$$
\end{lem}
\begin{proof}
	The lemma is from
	Feng \cite[Lemma 5.2.6]{Z}.
\end{proof}
Combining H\"{o}lder's inequality, Hua's lemma with Lemmas 4.2  and 4.3, we obtain
\begin{align}
     &\bigintsss_{{\mathfrak{m}}_1 }\left |S_2({\lambda_1 \alpha  ,\rho})\prod\limits_{j=2}^5 S_3({\lambda_j \alpha }) S_4({\lambda_6 \alpha }) S_k({\lambda_7 \alpha})\right | ^2K_\tau (\alpha )d\alpha \nonumber\\
\ll&\bigg(\underset{\alpha \in {\mathfrak{m}}_1  }{\sup }\left | S_2({\lambda_1 \alpha ,\rho}) \right |  \bigg)^{\frac{3}{2}+\frac{4}{m_2(k)} }\bigg(\bigintsss_{\mathbb{R} }\left | S_2({\lambda_1 \alpha ,\rho }) \right | ^4K_\tau (\alpha )d\alpha \bigg)^{\frac{1}{8} -\frac{1}{m_2(k)} }\nonumber\\
      &\times\bigg( \bigintsss_{\mathbb{R} }\left | S_3(\lambda _2\alpha )  \right | ^{4}\left | S_k(\lambda _7\alpha )  \right | ^{m_2(k)}K_\tau (\alpha )d\alpha \bigg)^{\frac{2}{m_2(k)} }\bigg(\bigintsss_{\mathbb{R} }\left | S_4({\lambda_6 \alpha}) \right | ^{16} K_\tau (\alpha )d\alpha \bigg)^\frac{1}{8} \nonumber\\  
      &\times\bigg(\bigintsss_{\mathbb{R} }\left | S_3({\lambda_2 \alpha  }) \right | ^8K_\tau (\alpha )d\alpha \bigg)^{\frac{1}{4} -\frac{1}{m_2(k)} }\prod\limits_{j=3}^5 \bigg(\bigintsss_{\mathbb{R} }\left | S_3({\lambda_j \alpha}) \right | ^{12} K_\tau (\alpha )d\alpha \bigg)^\frac{1}{6} \nonumber\\
 \ll&\tau X^{\frac{13}{6}+\frac{2}{k}+\frac{75}{84}-\frac{2}{7m_2(k)} +\varepsilon } .
\end{align}

Similarly, we have
\begin{align}
     &\bigintsss_{{\mathfrak{m}}_2 }\left |S_2({\lambda_1 \alpha  ,\rho})\prod\limits_{j=2}^5 S_3({\lambda_j \alpha }) S_4({\lambda_6 \alpha }) S_k({\lambda_7 \alpha})\right | ^2K_\tau (\alpha )d\alpha \nonumber\\
\ll& \bigg(\underset{\alpha \in {\mathfrak{m}}_2  }{\sup }\left | S_3({\lambda_2 \alpha}) \right |  \bigg)^{\frac{8}{5}+\frac{12}{m_1(k)} }
\bigg(\bigintsss_{\mathbb{R} }\left | S_3({\lambda_2 \alpha}) \right | ^6K_\tau (\alpha )d\alpha \bigg)^{\frac{1}{15}-\frac{2}{m_1(k)}} \nonumber\\
    &\times\bigg(\bigintsss_{\mathbb{R} }\left | S_k({\lambda_7 \alpha   }) \right |^{m_1(k)} K_\tau (\alpha )d\alpha \bigg)^\frac{2}{m_1(k)}\bigg(\bigintsss_{\mathbb{R} }\left | S_2({\lambda_1 \alpha  ,\rho  }) \right | ^6K_\tau (\alpha )d\alpha \bigg)^{ \frac{1}{3} } \nonumber\\
    &\times\bigg(\bigintsss_{\mathbb{R} }\left | S_4({\lambda_6\alpha}) \right | ^{20}K_\tau (\alpha )d\alpha \bigg)^\frac{1}{10}\prod\limits_{j=3}^5 \bigg(\bigintsss_{\mathbb{R} }\left | S_3({\lambda_j \alpha}) \right | ^{12} K_\tau (\alpha )d\alpha \bigg)^\frac{1}{6} \nonumber\\
\ll&\tau X^{\frac{13}{6}+\frac{2}{k}+\frac{44}{45}-\frac{24}{25m_1(k)} +\varepsilon }.  
\end{align}

Therefore, combining (4.1) and (4.2), we get
\begin{eqnarray}
\bigintsss_{\mathfrak{m}_i }\left |S_2({\lambda_1 \alpha  ,\rho})\prod\limits_{j=2}^5 S_3({\lambda_j \alpha }) S_4({\lambda_6 \alpha }) S_k({\lambda_7 \alpha})\right | ^2K_\tau (\alpha )d\alpha \ll
X^{\frac{13}{6}+\frac{2}{k}+\sigma (k) +\varepsilon },
\end{eqnarray}
where $i=1,2$, $\sigma (k)$ is defined in (1.3).

Write $\mathfrak{m^*} =\mathfrak{m}\setminus (\mathfrak{m}_1\cup \mathfrak{m}_2 ) $. In order to establish the upper bound for the integral $\mathfrak{m}^*$, we use some ideas due to Harman \cite{G} to bound the measure of $ I(\tau ,\upsilon ,\mathfrak{m}^*,\rho)  $. Note that for any $\alpha \in \mathfrak{m}^* $, we have
$$
\left | S_2(\lambda_1 \alpha,\rho ) \right |> X^{\frac{1}{2}-\frac{1}{14} +2\varepsilon }, \enspace
\left |  S_3({\lambda_2 \alpha}) \right | > X^{\frac{1}{3}-\frac{1}{36}+2\varepsilon}.
$$
We divide $\mathfrak{m}^*$ into disjoint set $S(Z_1,Z_2,y)$, such that for $\alpha$ $\in$ $S(Z_1,Z_2,y)$, we have 
$$
Z_1< \left | S_2(\lambda_1 \alpha,\rho ) \right |\leq 2Z_1,\enspace
Z_2< \left | S_3({\lambda_2 \alpha}) \right |\leq 2Z_2,\enspace
y< \left | \alpha  \right | \le 2y.
$$
where $Z_1=2^{k_1}X^{\frac{3}{7}+2\varepsilon } ,\ Z_2=2^{k_2}X^{\frac{11}{36}+2\varepsilon }$ and $y = 2^r X^{- \frac{3}{4}-\varepsilon}$ for some non-negative integers $k_1$, $k_2$, $r$.
By \cite[Lemma 3.2]{Zhu}, we derive from $X^{\frac{1}{2}}\ge Z_1\ge X^{\frac{1}{2}-\frac{1}{14} +2\varepsilon } $ and $\left | S_2(\lambda_1 \alpha,\rho ) \right |>Z_1$ that there exist integers $a_1$ and $q_1$ satisfy 
\begin{eqnarray}
1\leq q_1\ll \bigg(\frac{X^{\frac{1}{2}+\varepsilon}}{Z_1} \bigg)^2, \enspace
\left | q_1\lambda _1\alpha -a_1 \right | \ll \bigg(\frac{X^{\frac{1}{2}+\varepsilon}}{Z_1}  \bigg)^2.
\end{eqnarray}
Similarly, we can get
$X^\frac{1}{3} \ge Z_2\ge X^{\frac{1}{3}-\frac{1}{36} +2\varepsilon },\enspace
 \left | S_3({\lambda_2 \alpha}) \right | > Z_2$.

\begin{lem}
Suppose that $X^\frac{1}{3} \ge Z_2\ge X^{\frac{1}{3}-\frac{1}{36} +2\varepsilon},\enspace
\left | S_3({\lambda_2 \alpha }) \right | > Z_2$. Then there are coprime integers $a_2, q_2$ satisfying
$$
1\leq q_2\ll \bigg(\frac{X^{\frac{1}{3} +\varepsilon }}{Z_2} \bigg)^2,\enspace
\left | q_2\lambda _2\alpha -a_2 \right | \ll X^{-1}\bigg(\frac{X^{\frac{1}{3} +\varepsilon}}{Z_2} \bigg)^2.
$$
\end{lem}
\begin{proof}
    The proof is quite similar as \cite [Corollary 2.2]{Ge2} .
\end{proof}

We define $a_1a_2\ne 0$, otherwise we assume that there exists $\alpha \in \mathfrak{M} $. In addition, we subdivide $S(Z_1,Z_2,y) $ into sets 
$ S(Z_1,Z_2,y,Q_1,Q_2)$, where $ Q_j< q_j\leq 2Q_j $ on each set. Since
$$
a_2q_1\frac{\lambda _1}{\lambda _2} -a_1q_2=(q_1\lambda _1\alpha -a_1)\frac{a_2}{\lambda _2\alpha }-(q_2\lambda _2\alpha-a_2)\frac{a_1}{\lambda _2\alpha} ,
$$
then it follows from (4.4) and Lemma 4.4 that
$$
\left | a_2q_1\frac{\lambda _1}{\lambda _2}-a_1q_2 \right | \ll X^{\frac{2}{3}+4\varepsilon} Z_1^{-2}Z_2^{-2}\ll X^{-\frac{101}{126} -\varepsilon}.
$$
Recall that $ q=X^{\frac{101}{126} } $. Thus
$$
\left | a_2q_1\frac{\lambda _1}{\lambda _2}-a_1q_2 \right | = o (q^{-1}).
$$
We  also can get
$$
\left | a_2q_1 \right | \ll yQ_1Q_2.
$$
Therefore, if $\left | a_2q_1 \right |$ took $W$ distinct values, we could deduce the existence of $n $ satisfying
$$\left \| n\frac{\lambda _1}{\lambda _2}  \right \| \ll X^{-\frac{101}{126} -\varepsilon},\enspace
n\ll \frac{yQ_1Q_2}{R} .
$$
This would contradict $ a/q $ being a convergent to $\lambda_1/\lambda _2$ if $q$ is sufficiently large, unless 
$$
W\ll  \frac{yQ_1Q_2}{q}.
$$
By the upper bound for the divisor function, we observe that each value of $\left | a_2q_1 \right | $  corresponds to $O(X^\varepsilon )$  values of $a_2 $ and $q_1$. Therefore, each set in $ S(Z_1,Z_2,y,Q_1,Q_2)$ consists of $O(WX^\varepsilon )$ length intervals
\begin{align}
&\ll \min \bigg(Q_1^{-1}X^{-1} \bigg( \frac{X^{\frac{1}{2}+\varepsilon}}{Z_1}\bigg)^2,\enspace Q_2^{-1}X^{-1} \bigg( \frac{X^{\frac{1}{3}+\varepsilon}}{Z_2}\bigg)^2\bigg)
\notag\\
&\ll \frac{X^{-\frac{1}{6}+\varepsilon}}{Z_1Z_2Q_1^{\frac{1}{2}}Q_2^{\frac{1}{2}}}.\notag
\end{align}
Let $\mathfrak{L} $ denote such a set $ S(Z_1,Z_2,y,Q_1,Q_2)$. Note that
\begin{eqnarray}
\bigintsss_{\mathfrak{L} }d\alpha &\ll& yQ_1Q_2q^{-1}\frac{X^{-\frac{1}{6}+\varepsilon}}{Z_1Z_2Q_1^{\frac{1}{2}}Q_2^{\frac{1}{2}}} \notag\\
      &\ll& \frac{yq^{-1}X^{-\frac{1}{6}+\varepsilon}Q_1^{\frac{1}{2}}Q_2^{\frac{1}{2}}}{Z_1Z_2}.\notag
\end{eqnarray}
Then integration over the set $\mathfrak{L} $ gives
\begin{eqnarray}
     &&\bigintsss_{\mathfrak{L} }\left |S_2({\lambda_1 \alpha  ,\rho})\prod\limits_{j=2}^5 S_3({\lambda_j \alpha }) S_4({\lambda_6 \alpha }) S_k({\lambda_7 \alpha})\right | ^2K_\tau (\alpha )d\alpha \notag\\
   &\ll& \min (\tau ^2,y^{_{-2}}) Z_1^{2} Z_2^{2}X^{2}X^{\frac{1}{2}} X^{\frac{2}{k} }
    \bigg(\bigintsss_{\mathfrak{L} }d\alpha  \bigg) \notag\\
   &\ll& \tau y^{-1} Z_1^{2} Z_2^{2}X^{\frac{5}{2}+\frac{2}{k}} \frac{yq^{-1}X^{-\frac{1}{6}+\varepsilon}Q_1^{\frac{1}{2}}Q_2^{\frac{1}{2}}}{Z_1Z_2}.\notag
\end{eqnarray}  
Recall that
\begin{eqnarray}
	q=X^\frac{101}{126} ,\enspace Q_1\ll \bigg(\frac{X^{\frac{1}{2}+\varepsilon}}{Z_1} \bigg)^2, \enspace Q_2\ll \bigg(\frac{X^{\frac{1}{3}+\varepsilon  }}{Z_2} \bigg)^2.\notag
\end{eqnarray}
Thus
$$
\bigintsss_{\mathfrak{L} }\left |S_2({\lambda_1 \alpha  ,\rho})\prod\limits_{j=2}^5 S_3({\lambda_j \alpha }) S_4({\lambda_6 \alpha }) S_k({\lambda_7 \alpha})\right | ^2K_\tau (\alpha )d\alpha
\ll \tau X^{\frac{13}{6}+\frac{2}{k} +\frac{15}{126}+3\varepsilon}.
$$
From a dyadic dissection argument, we obtain
\begin{eqnarray}
\bigintsss_{\mathfrak{m}^* }\left |S_2({\lambda_1 \alpha  ,\rho})\prod\limits_{j=2}^5 S_3({\lambda_j \alpha }) S_4({\lambda_6 \alpha }) S_k({\lambda_7 \alpha})\right | ^2K_\tau (\alpha )d\alpha 
   &\ll& (\log{X} )^5 \tau X^{\frac{13}{6}+\frac{2}{k}+\frac{15}{126}+\frac{7}{4}\varepsilon }\notag\\
   &\ll& \tau X^{\frac{13}{6}+\frac{2}{k} +\sigma (k)+\varepsilon}.
\end{eqnarray}
Based on the comprehensive analysis of the existing results (4.3) and (4.5), we have 
\begin{eqnarray}
\bigintsss_{\mathfrak{m} }\left |S_2({\lambda_1 \alpha  ,\rho})\prod\limits_{j=2}^5 S_3({\lambda_j \alpha }) S_4({\lambda_6 \alpha }) S_k({\lambda_7 \alpha})\right | ^2K_\tau (\alpha )d\alpha
  &\ll& \tau X^{\frac{13}{6}+\frac{2}{k} +\sigma (k)+\varepsilon}.
\end{eqnarray}

\section{The trivial arc}

In this section, we estimate the contribution of the trivial arc. It follows from H\"{o}lder's inequality and (2.1) that 
\begin{eqnarray}
I(\tau ,\upsilon,\mathfrak{t},\rho)
   &\ll&\bigintsss_{\mathfrak{t} }\left|S_2(\lambda _1\alpha ,\rho )\prod\limits_{j=2}^5S_3(\lambda _j\alpha )S_4(\lambda _6\alpha )S_k(\lambda _7\alpha )\right|K_\tau (\alpha )d\alpha  \notag\\
   &\ll&\left(\bigintsss_{\mathfrak{t} }\left| S_2({\lambda_1 \alpha  ,\rho  }) \right| ^4 K_\tau (\alpha )d\alpha \right)^{\frac{1}{4}}\prod\limits_{j=2}^5\left( \bigintsss_{\mathfrak{t} }\left | S_3({\lambda_2 \alpha  ,\rho  }) \right | ^8 K_\tau (\alpha )d\alpha \right)^{\frac{1}{8}} \notag\\
    & &\times\left ( \bigintsss_{\mathfrak{t} }\left | S_4({\lambda_6 \alpha   }) \right |^8 K_\tau (\alpha )d\alpha \right )^\frac{1}{8}\left ( \bigintsss_{\mathfrak{t} }\left | S_k({\lambda_7 \alpha   }) \right |^8 K_\tau (\alpha )d\alpha \right )^\frac{1}{8}
    \notag\\
   &\ll&\left( \sum_{n=\left [ R \right ] }^{\infty }  \bigintsss_{n}^{n+1} \frac{\left | S_2({\lambda_1 \alpha  ,\rho  }) \right |^4 }{\left | \alpha  \right | ^2}d\alpha  \right)  ^\frac{1}{4}\prod\limits_{j=2}^5\left(\sum_{n=\left [ R \right ] }^{\infty }  \bigintsss_{n}^{n+1} \frac{\left |S_3({\lambda_j \alpha}) \right |^8 }{\left | \alpha  \right |^2}d\alpha \right)^{\frac{1}{8}}\notag\\
   & &\times \left( \sum_{n=\left [ R \right ] }^{\infty }  \bigintsss_{n}^{n+1} \frac{\left | S_4({\lambda_6 \alpha}) \right |^8 }{\left | \alpha  \right | ^2}d\alpha \right)^\frac{1}{8}\left ( \sum_{n=\left [ R \right ] }^{\infty }  \bigintsss_{n}^{n+1} \frac{\left | S_k({\lambda_7 \alpha   })\right |^8 }{\left | \alpha  \right | ^2}d\alpha  \right)  ^\frac{1}{8}.\notag
\end{eqnarray}
 Using Hua's lemma, we have
\begin{eqnarray}
 I(\tau ,\upsilon ,\mathfrak{t},\rho) &\ll&\left ( \sum_{n=\left [ R \right ] }^{\infty }\frac{1}{n^2}   \right ) 
\left ( \bigintsss_{0}^{1}  \left | S_2({\lambda_1 \alpha  ,\rho  }) \right |^4 d\alpha \right ) ^\frac{1}{4}\prod\limits_{j=2}^5\left(  \bigintsss_{0}^{1} \left |S_3({\lambda_j \alpha}) \right |^8 d\alpha \right)^{\frac{1}{8}}
 \notag\\
 & &\times\left( \bigintsss_{0}^{1}\left | S_4({\lambda_6 \alpha}) \right |^8 d\alpha \right)^\frac{1}{8}
\left ( \bigintsss_{0}^{1}  \left | S_k({\lambda_3 \alpha   }) \right |^8 d\alpha\right ) ^\frac{1}{8} \notag\\
  &\ll&R^{-1}X^{\frac{13}{12} +\frac{5}{32} +\frac{5}{8k}+\varepsilon } \notag\\
  &\ll&\tau ^2X^{\frac{13}{12} +\frac{1}{k} +\varepsilon }. \notag
\end{eqnarray}
Therefore, we can get
$$
R=\tau ^{-2}X^{\frac{5k-12}{32k} }L^3, 
$$
and we obtain
\begin{eqnarray}
	\left | I(\tau ,\upsilon ,\mathfrak{t},\rho)\right | = o \left ( {\tau}^2 L^{-1} X^{\frac{13}{12} +\frac{1}{k}} \right).
\end{eqnarray}

\section{Completion of the proof}
In this section, we complete the proof of Theorem 1.1. Let $\mathcal{E} _k$$(\mathcal {V}  ,X,\delta  ) $ be the set $v\in \mathcal{V}$ with  $v\leq X $ such that there exists no solution in primes $p_1$, $p_2$, $p_3$, $p_4$, $p_5$, $p_6$, $p_7$ satisfying inequality (1.2). Note that $E_k(\mathcal {V}  ,X, \delta  )=\left | \mathcal{E} _k(\mathcal {V},X, \delta  ) \right | $. Let $Y=\left [  \frac{1}{2} X,X\right ] $, $\mathcal{E} _k(Y)=\mathcal{E}_k(\mathcal{V},X,\delta  )\cap Y
 $ and $E_k(Y)=\left | \mathcal{E} _k(Y) \right |$.
By (3.8) and (5.1), we obtain
\begin{eqnarray}
\left | \sum_{v\in \mathcal{E}_k(Y)}\bigintsss_{\mathfrak{m} }S_2({\lambda_1 \alpha  ,\rho})\prod\limits_{j=2}^5 S_3({\lambda_j \alpha }) S_4({\lambda_6 \alpha }) S_k({\lambda_7 \alpha})K_\tau (\alpha )e(-\upsilon\alpha)d\alpha \right |\gg {\tau}^2 L^{-1} X^{\frac{13}{12} +\frac{1}{k}}E_k(Y) .\notag\\
\end{eqnarray}
By Cauchy's inequality, (2.2) and (4.6), we have
\begin{eqnarray}\label{bbb}
 & &\left | \sum_{v\in \mathcal{E}_k(Y)}\bigintsss_{\mathfrak{m} }S_2({\lambda_1 \alpha  ,\rho})\prod\limits_{j=2}^5 S_3({\lambda_j \alpha }) S_4({\lambda_6 \alpha }) S_k({\lambda_7 \alpha})K_\tau (\alpha )e(-\upsilon\alpha)d\alpha \right |\notag\\
&\ll&  \bigg(\bigintsss_{\mathfrak{m} } \left |S_2({\lambda_1 \alpha  ,\rho})\prod\limits_{j=2}^5 S_3({\lambda_j \alpha }) S_4({\lambda_6 \alpha }) S_k({\lambda_7 \alpha})\right |^2K_\tau (\alpha )d\alpha \bigg  )^\frac {1}{2} \notag \\
 & &\times\bigg(\bigintsss_{\mathbb{R} } \bigg| \sum_{v\in \mathcal{E}_k(Y) }e(-\alpha v)  \bigg|  ^2K_\tau(\alpha )d\alpha  \bigg)^{\frac{1}{2} }\notag \\
&\ll&\bigg{(\tau X^{\frac{13}{6}+\frac{2}{k} +\sigma (k)+\varepsilon  } }\bigg)^{\frac{1}{2} }
\bigg(\sum_{v_1,v_2\in \mathcal{E}_k(Y)}\max(0,\tau -\left | v_1-v_2 \right | ) \bigg)^{\frac{1}{2} } \notag \\
&\ll& \tau (E_k(Y))^{\frac{1}{2} } \bigg( X^{\frac{13}{6}+\frac{2}{k} +\sigma (k)+\varepsilon  } \bigg)^{\frac{1}{2} }.
\end{eqnarray} 
Then combining (6.1) and (6.2), we get
\begin{eqnarray}
E_k(Y)\ll \tau ^{-2}X^{\sigma (k)+\varepsilon }.
\end{eqnarray}
Finally, we use the stand dyadic argument. We take $\tau =X^{-\delta } $ and denote a well-spaced sequence by $\mathcal{V} $ for $v_i\in \mathcal{V}$  for $i=1,2,\cdots$, according to (6.3), by the definition of $\mathcal{E}_k(\mathcal{V},X,\delta  )$, it is easy to show that
$$
E_k(\mathcal{V},X,\delta  )\ll X^{\sigma (k)+2\delta +\varepsilon }.
$$
Since $\lambda _1/{\lambda _2}$  is irrational, there exists an infinite set of integer pairs $q,a$. Then we have $X=q^{\frac{126}{101} }\to +\infty$, as $q\to +\infty $. Thus, the proof of Theorem 1.1 is concluded.
\section*{Acknowledgements}

This work is supported by Natural Science Foundation of China (Grant Nos. 12361002) and Natural Science Foundation of Jiangxi Province (Grant Nos. 20224BAB201001). The authors would like to express their thanks to the referee suggestions and comments on the manuscript.


\begin{thebibliography}{99}
\bibitem{DH}
 Davenport, H.,  Heilbronn, H.: On indefinite quadratic forms in five variables. J. Lond. Math. Soc. 21, 185-193 (1946)

\bibitem{Z}
 Feng, Z.Z.: The exceptional sets for Waring-Goldbach problem, Ph.D. Thesis, Jilin University, Jilin, China (2019)
 
\bibitem{Ge}
 Ge, W.X.,  Zhao, F.: The exceptional set for Diophantine inequality with unlike powers of prime variables. Czech.
Math. J. 68 (1), 149-168 (2018)

\bibitem{Ge2}
 Ge, W.X., Zhao, F.: The values of cubic forms at prime arguments. J. Number Theory. 180, 694-709 (2017)
 
\bibitem{G}
 Harman, G.: The values of ternary quadratic forms at prime arguments. Math. 51, 83-96 (2005)
\bibitem{888}
Harman, G.: Trigonometric sums over primes I. Math. 28, 249-254 (1981)
\bibitem{H}
 Harman, G., Kumchev, A.V.: On sums of squares of primes.  Math. Proc. Cambridge Philos. Soc. 140 (1), 1-13 (2006)
 
\bibitem{666}
Languasco, A., Zaccagnini, A.: On a ternary Diophantine problem with mixed powers of primes. Acta Arith. 159, 345-362 (2013)

\bibitem{W}
 Li, B.Y., Wang, Y.C.: Diophantine approximation with prime variables and mixed powers. Ramanujan J. 60, 371-389 (2023)

\bibitem{lwp}
 Li, W.P.: One additive Diophantine inequality with mixed powers. Journal of Qufu Normal University (Natural Science) 31 (2), 39-42 (2005)
 
\bibitem{lwp2}
 Li, W.P., Gong, K.: Diophantine inequality with mixed powers 2, 3, a and b (in Chinese). Pure and Applied Mathematics 25 (3), 497-501 (2009)

\bibitem{qyy}
 Qu, Y.Y.,  Zeng, J.W.: Diophantine approximation with prime variables and mixed powers. Ramanujan
J. 52, 625-639 (2020)

\bibitem{T}
  Titchmarsh, E.C.: The Theory of the Riemann Zeta-Function, 2nd edn. Oxford University Press, Oxford (1986)

\bibitem{xly}
Xi, L.Y., Mu, Q.W.: An additive Diophantine inequality with mixed powers. Acta Math. Sinica, Chin. Ser. 67 (1), 187-194 (2024)
\bibitem{yhb}
 Yu, H.B.: Diophantine inequalities with mixed powers (II).  Acta Math. Sinica, Chin. Ser. 37 (3), 324-331 (1994)
 \bibitem{Zhu}
 Zhu, L.: Diophantine inequality by unlike powers of primes. Chin. Ann. Math. Ser. B 43 (1), 125-136 (2022)
\end{thebibliography}
\end{document}